\documentclass[english]{amsart}
\usepackage[T1]{fontenc}
\usepackage[latin9]{inputenc}
\usepackage{amsthm}
\usepackage{amssymb}

\makeatletter
\numberwithin{equation}{section}
\numberwithin{figure}{section}
\theoremstyle{plain}
\newtheorem{thm}{\protect\theoremname}
\theoremstyle{remark}
\newtheorem{rem}[thm]{\protect\remarkname}
\theoremstyle{plain}
\newtheorem{cor}[thm]{\protect\corollaryname}
\theoremstyle{plain}
\newtheorem{lem}[thm]{\protect\lemmaname}

\makeatother

\usepackage{babel}
\providecommand{\corollaryname}{Corollary}
\providecommand{\lemmaname}{Lemma}
\providecommand{\remarkname}{Remark}
\providecommand{\theoremname}{Theorem}

\begin{document}
\title{On Zeros of Certain Entire Functions}
\author{Ruiming Zhang}
\email{ruimingzhang@guet.edu.cn}
\address{School of Mathematics and Computing Sciences\\
Guilin University of Electronic Technology\\
Guilin, Guangxi 541004, P. R. China. }
\thanks{This work is supported by the National Natural Science Foundation
of China, grant No. 11771355.}
\begin{abstract}
In this work we derive a sufficient condition to ensure certain genus
0 entire function that can have only negative zeros. We also apply
this result to the Riemann hypothesis and generalized Riemann hypothesis
for some primitive Dirichlet character. 
\end{abstract}

\maketitle

\section{\label{sec:Intro} Introduction }

In this work we study the zero distributions of certain genus $0$
entire functions. Specifically we prove the following result:
\begin{thm}
\label{thm:1} Let $f(z)$ be an entire function with $f(0)\neq0$
such that
\begin{equation}
\frac{f(z)}{f(0)}=\prod_{n=1}^{\infty}\left(1+\frac{z}{\lambda_{n}}\right)=\sum_{n=0}^{\infty}a_{n}z^{n},\label{eq:1.1}
\end{equation}
 where 
\begin{equation}
\sum_{n=1}^{\infty}\frac{1}{\left|\lambda_{n}\right|}<\infty,\quad a_{0}=1,\ \left\{ a_{n}\right\} _{n=1}^{\infty}\subset(0,\infty).\label{eq:1.2}
\end{equation}
If there exist positive numbers $\alpha_{0},\beta_{0}$ with $\alpha_{0},\beta_{0}\in(0,1)$
such that 
\begin{equation}
\sum_{n=1}^{\infty}\frac{1}{\left|\lambda_{n}\right|^{\alpha_{0}}}<\infty\label{eq:1.3}
\end{equation}
and $\forall n\in\mathbb{N}$,
\begin{equation}
\Re(\lambda_{n})\ge\beta_{0}\left|\lambda_{n}\right|>0,\label{eq:1.4}
\end{equation}
then all the zeros of $f(z)$, i.e. $\left\{ -\lambda_{n}\right\} _{n=1}^{\infty}$,
are negative.
\end{thm}

\begin{rem}
\label{rem:2} Notice that if the order $\rho(f)$ of $f(z)$ is strictly
less than $1$, then by \cite[Theorem 2.5.18]{Boas} the condition
(\ref{eq:1.3}) holds for any $\alpha_{0}>\rho(f)$.
\end{rem}

Now we apply Theorem \ref{thm:1} to we get the following:
\begin{cor}
\label{cor:3} Let 
\begin{equation}
g(z)=\sum_{n=0}^{\infty}a_{n}(-z^{2})^{n},\quad a_{0}\neq0,\ \frac{a_{n}}{a_{0}}>0,\ \forall n\in\mathbb{N}\label{eq:1.5}
\end{equation}
be an even entire function of order strictly less $2$ with at least
one root, and all of its nonzero roots be $\left\{ \pm z_{n}\vert n\in\mathbb{N}\right\} $
with $\Re(z_{n})>0$. If there exists a positive number $M$ such
that
\begin{equation}
\left|\Im(z_{n})\right|\le M,\quad\forall n\in\mathbb{N},\label{eq:1.6}
\end{equation}
then all the roots $\left\{ \pm z_{n}\vert n\in\mathbb{N}\right\} $
of $g(z)$ are real.
\end{cor}

\begin{proof}
Clearly, 
\begin{equation}
f(z)=g(i\sqrt{z})=\sum_{n=0}^{\infty}a_{n}z^{n}\label{eq:1.7}
\end{equation}
is an entire function of order strictly less than $1$, and $f(-z^{2})=0$
if and only if $g(z)=g(-z)=0$. Hence, $f(z)$ has only negative zeros
$\left\{ -z_{n}^{2}\right\} _{n=1}^{\infty}$ if and only if $g(z)$
has only real zeros $\left\{ \pm z_{n}\vert n\in\mathbb{N}\right\} $.

Since the order of $f(z)$ is strictly less than $1$, then by Remark
\ref{rem:2} the condition (\ref{eq:1.3}) holds with $\lambda_{n}=z_{n}^{2},\ \forall n\in\mathbb{N}$. 

Since $\left\{ -z_{n}\right\} $ has no limit point on the finite
complex plane, then we may assume that 
\begin{equation}
0<\Re(z_{1})\le\Re(z_{2})<\dots\le\Re(z_{n})\le\Re(z_{n+1})\le\dots.\label{eq:1.8}
\end{equation}
Since we also have
\begin{equation}
\lim_{N\to\infty}\left|z_{n}\right|=+\infty,\label{eq:1.9}
\end{equation}
then for any $\epsilon\in(0,1)$ there exists a $N_{\epsilon}\in\mathbb{N}$
such that
\begin{equation}
\epsilon\left|z_{n}\right|^{2}>2M^{2},\quad\forall n\ge N_{\epsilon}.\label{eq:1.10}
\end{equation}
Hence,
\begin{equation}
\Re(z_{n}^{2})=\left|z_{n}\right|^{2}-2\left(\Im(z_{n})\right)^{2}\ge\left|z_{n}\right|^{2}-2M^{2}\ge(1-\epsilon)\left|z_{n}\right|^{2}.\label{eq:1.11}
\end{equation}
Let 
\begin{equation}
\beta_{0}=\min\left\{ 1-\epsilon,\frac{\Re(z_{n}^{2})}{\left|z_{n}^{2}\right|},\ 1\le n\le N_{\epsilon}\right\} ,\label{eq:1.12}
\end{equation}
then $\beta_{0}\in(0,1)$ and
\begin{equation}
\Re(z_{n}^{2})\ge\beta_{0}\left|z_{n}^{2}\right|,\quad\forall n\in\mathbb{N}.\label{eq:1.13}
\end{equation}
 Then the condition (\ref{eq:1.4}) holds for $\lambda_{n}=z_{n}^{2},\ \forall n\in\mathbb{N}$.
Applying Theorem \ref{thm:1} we proved that $f(z)$ has only negative
zeros $\left\{ -z_{n}^{2}\right\} _{n=1}^{\infty}$, which is equivalent
to that $g(z)$ has only real zeros $\left\{ \pm z_{n}\vert n\in\mathbb{N}\right\} $.
\end{proof}

\section{\label{sec:proof}Proof of Theorem \ref{thm:1}}

\subsection{Preliminaries}
\begin{lem}
\label{lem:4} Let $f(z)$ be any entire function defined as in Theorem
\ref{thm:1}, we define its associated heat kernel by

\begin{equation}
\Theta(t\vert f)=\sum_{n=1}^{\infty}e^{-\lambda_{n}t},\quad t>0.\label{eq:2.1}
\end{equation}
Then the real function $\Theta(t\vert f)$ is in $C^{\infty}(0,\infty)$.
Furthermore, for any 
\begin{equation}
k\in\mathbb{N}_{0},\ 0<\beta<\ell=\inf\left\{ \Re(\lambda_{n})\bigg|n\in\mathbb{N}\right\} ,\label{eq:2.2}
\end{equation}
the function $\Theta(t\vert f)$ satisfies 

\begin{equation}
\Theta^{(k)}(t\vert f)=\mathcal{O}\left(t^{-\alpha_{0}-k}\right),\quad t\to0^{+},\label{eq:2.3}
\end{equation}
\begin{equation}
\Theta^{(k)}(t\vert f)=\mathcal{O}\left(e^{-\beta t}\right),\quad t\to+\infty\label{eq:2.4}
\end{equation}
and for any $z\in\mathbb{C}$ with $\Re(z)\ge0$,
\begin{equation}
\int_{0}^{\infty}t^{k}e^{-\Re(z)t}\left|\Theta(t\vert f)\right|dt\le\frac{k!}{\beta_{0}^{k+1}}\sum_{n=1}^{\infty}\frac{1}{\left|\lambda_{n}\right|^{k+1}}<\infty\label{eq:2.5}
\end{equation}
\textup{and}
\begin{equation}
(-1)^{k}\left(\frac{f'(z)}{f(z)}\right)^{(k)}=\int_{0}^{\infty}t^{k}e^{-zt}\Theta(t\vert f)dt.\label{eq:2.6}
\end{equation}

\end{lem}

\begin{proof}
For any $t>0,\,k\ge0$ by (\ref{eq:1.3}), (\ref{eq:1.4}) and
\begin{equation}
\sup_{x>0}x^{a}e^{-x}=\left(\frac{a}{e}\right)^{a},\quad a>0,\label{eq:2.7}
\end{equation}
 we get
\begin{equation}
\begin{aligned} & \sum_{n=1}^{\infty}\left|\lambda_{n}^{k}e^{-\lambda_{n}t}\right|=\sum_{n=1}^{\infty}\left|\lambda_{n}\right|^{k}e^{-\Re(\lambda_{n})t}\le\sum_{n=1}^{\infty}\frac{\left(\beta_{0}t\left|\lambda_{n}\right|\right)^{k+\alpha_{0}}e^{-\beta_{0}\left|\lambda_{n}\right|t}}{(\beta_{0}t)^{\alpha_{0}+k}\left|\lambda_{n}\right|^{\alpha_{0}}}\\
 & \le\frac{\sup_{x>0}x^{k+\alpha_{0}}e^{-x}}{(\beta_{0}t)^{\alpha_{0}+k}}\sum_{n=1}^{\infty}\frac{1}{\left|\lambda_{n}\right|^{\alpha_{0}}}=\left(\frac{k+\alpha_{0}}{e\beta_{0}t}\right)^{k+\alpha_{0}}\sum_{n=1}^{\infty}\frac{1}{\left|\lambda_{n}\right|^{\alpha_{0}}}<\infty,
\end{aligned}
\label{eq:2.8}
\end{equation}
which proves that $\Theta(t\vert f)\in C^{\infty}(0,\infty)$ and
(\ref{eq:2.3}). 

By (\ref{eq:1.1}) and (\ref{eq:1.2}) the roots of $f(z)$ must appear
in complex conjugate pairs, therefore $\Theta(t\vert f)$ is a real
function in $t\in(0,\infty)$.

Since $0<\beta<\ell$, then there exists a positive number $\epsilon$
with $0<\epsilon<1$ such that 
\begin{equation}
\beta=\epsilon\ell=\epsilon\inf\left\{ \Re(\lambda_{n})\bigg|n\in\mathbb{N}\right\} \le\epsilon\Re(\lambda_{n}),\quad\forall n\in\mathbb{N}.\label{eq:2.9}
\end{equation}
Thus for $t\ge\frac{1}{1-\epsilon}$ in (\ref{eq:2.8}),
\begin{equation}
\begin{aligned} & \left|e^{\beta t}\Theta^{(k)}(t\vert f)\right|\le\sum_{n=1}^{\infty}\left|\lambda_{n}\right|^{k}e^{-(1-\epsilon)\Re(\lambda_{n})t}\\
 & \le\sum_{n=1}^{\infty}\left|\lambda_{n}\right|^{k}e^{-(1-\epsilon)\beta_{0}\left|\lambda_{n}\right|t}\le\sum_{n=1}^{\infty}\left|\lambda_{n}\right|^{k}e^{-\beta_{0}\left|\lambda_{n}\right|}<\infty,
\end{aligned}
\label{eq:2.10}
\end{equation}
which establishes (\ref{eq:2.4}). 

For all $k\ge0$ and $z\in\mathbb{C}$ with $\Re(z)\ge0$ since 
\begin{equation}
\begin{aligned} & \int_{0}^{\infty}t^{k}e^{-\Re(z)t}\left|\Theta(t\vert f)\right|dt\le\int_{0}^{\infty}t^{k}\left(\sum_{n=1}^{\infty}e^{-\Re(\lambda_{n})t}\right)dt\\
 & \le\sum_{n=1}^{\infty}\frac{k!}{\left(\Re(\lambda_{n})\right)^{k+1}}\le\frac{k!}{\beta_{0}^{k+1}}\sum_{n=1}^{\infty}\frac{1}{\left|\lambda_{n}\right|^{k+1}}<\infty,
\end{aligned}
\label{eq:2.11}
\end{equation}
which gives \ref{eq:2.5}. The equation (\ref{eq:2.6}) is obtained
by applying the Fubini's theorem.
\end{proof}

\subsection{Proof of the Theorem \ref{thm:1}}
\begin{proof}
Assume Theorem \ref{thm:1} is false, i.e. not all the zeros of the
entire function $f(z)$ are inside $(-\infty,0)$. Clearly, it is
equivalent to not all the poles of $\frac{f'(z)}{f(z)}$ are inside
$(-\infty,0)$. 

Let $\left\{ \mu_{n_{m}}\right\} _{m=1}^{\infty}$ be the sequence
containing all the positive element of $\left\{ \lambda_{n}\right\} _{n=1}^{\infty}$,
then the sequence $\left\{ \mu_{n_{m}}\right\} _{m=1}^{\infty}$ also
satisfies (\ref{eq:1.3}) and (\ref{eq:1.4}). 

Let
\begin{equation}
h(z)=\prod_{m=1}^{\infty}\left(1+\frac{z}{\mu_{n_{m}}}\right)=\sum_{n=0}^{\infty}c_{n}z^{n},\label{eq:2.12}
\end{equation}
where $c_{0}=1,\ c_{n}>0,\forall n\in\mathbb{N}$ and
\begin{equation}
\Theta(t\vert h)=\sum_{m=1}^{\infty}e^{-\mu_{n_{m}}t},\quad\forall t>0,\label{eq:2.13}
\end{equation}
then they satisfy properties (\ref{eq:2.2}) to (\ref{eq:2.6}). In
particular, for any $z\in\mathbb{C}$ with $\Re(z)\ge0$,
\begin{equation}
\frac{h'(z)}{h(z)}=\int_{0}^{\infty}e^{-zt}\Theta(t\vert h)dt.\label{eq:2.14}
\end{equation}

Let 
\begin{equation}
g(z)=\frac{f'(z)}{f(z)}-\frac{h'(z)}{h(z)},\label{eq:2.15}
\end{equation}
then $g(z)$ is an meromorphic function has all the poles $-\lambda_{n}$
of $\frac{f'(z)}{f(z)}$ such that $\lambda_{n}$ are complex. Therefore,
$g(z)$ must have at least a pair of simple poles at $-\lambda_{n_{0}}=re^{i\theta}$
and $-\overline{\lambda_{n_{0}}}=re^{-i\theta}$ such that
\begin{equation}
r>0,\ \frac{\pi}{2}<\theta<\pi,\quad\Re(\lambda_{n_{0}})=r\cos\theta<0.\label{eq:2.16}
\end{equation}
Notice that $g(z)$ is analytic inside 
\begin{equation}
\left\{ z:\Re(z)>-\Re(\lambda_{1})\right\} \cup\left\{ z:\left|\arg(z)\right|\le\frac{\pi}{2}\right\} .\label{eq:2.17}
\end{equation}
Since by Lemma \ref{lem:4} we have

\begin{equation}
g(x)=\int_{0}^{\infty}e^{-xt}\left(\Theta(t\vert f)-\Theta(t\vert h)\right)dt,\quad x\ge0\label{eq:2.18}
\end{equation}
 and $\forall k\in\mathbb{N}_{0}$ and $\forall\Re(z)\ge0$,
\begin{equation}
\int_{0}^{\infty}t^{k}e^{-\Re(z)t}\left|\Theta(t\vert f)-\Theta(t\vert h)\right|dt\le\frac{k!}{\beta_{0}^{k+1}}\sum_{n=1}^{\infty}\frac{1}{\left|\lambda_{n}\right|^{k+1}}<\infty.\label{eq:2.19}
\end{equation}
Then,
\begin{equation}
\int_{0}^{\infty}e^{-zt}\left(\Theta(t\vert f)-\Theta(t\vert h)\right)dt\label{eq:2.20}
\end{equation}
is analytic on $\Re(z)>0$ and bounded on $\Re(z)\ge0$. Thus, 
\begin{equation}
g(z)=\int_{0}^{\infty}e^{-zt}\left(\Theta(t\vert f)-\Theta(t\vert h)\right)dt\label{eq:2.21}
\end{equation}
 on $\Re(z)\ge0$ by analytic continuation. Then 
\begin{equation}
g(z^{1/2})=\int_{0}^{\infty}e^{-z^{1/2}t}\left(\Theta(t\vert f)-\Theta(t\vert h)\right)dt\label{eq:2.22}
\end{equation}
 on $\left\{ z:\left|\arg(z)\right|<\pi\right\} $ and $g(z^{1/2})$
is bounded on $\left\{ z:\left|\arg(z)\right|\le\pi\right\} .$ 

On the other hand, since 
\begin{equation}
\pi<2\theta<2\pi,\quad-\pi>-2\theta>-2\pi,\label{eq:2.23}
\end{equation}
then both $\lambda_{0}^{2}=r^{2}e^{2i\theta}=r^{2}e^{i(2\theta-2\pi)}$
and $\overline{\lambda_{0}^{2}}=r^{2}e^{-2i\theta}=r^{2}e^{i(-2\theta+2\pi)}$
are in $\left\{ z:\left|\arg(z)\right|<\pi\right\} $. 

Since at poles $-\lambda_{n_{0}},\,-\overline{\lambda_{n_{0}}}$ $g\left(z\right)$
has the respective singular parts, 
\begin{equation}
\frac{R}{z+\lambda_{0}},\quad\frac{R}{z+\overline{\lambda_{0}}},\quad R>0.\label{eq:2.24}
\end{equation}
Then the function $g\left(z^{1/2}\right)$ has corresponding singular
parts at $\lambda_{0}^{2}$ and $\overline{\lambda_{0}^{2}}$ ,
\begin{equation}
\frac{-2R\lambda_{0}}{z-\lambda_{0}^{2}},\ \frac{-2R\overline{\lambda_{0}}}{z-\overline{\lambda_{0}^{2}}},\quad R\lambda_{0}\neq0.\label{eq:2.25}
\end{equation}
Since both $\lambda_{0}^{2}$ and $\overline{\lambda_{0}^{2}}$ in
$\left\{ z:\left|\arg(z)\right|<\pi\right\} $, then $g\left(z^{1/2}\right)$
is unbounded on $\left\{ z:\left|\arg(z)\right|\le\pi\right\} $. 

Since $g\left(z^{1/2}\right)$ can not both bounded and unbounded
on $\left\{ z:\left|\arg(z)\right|\le\pi\right\} ,$ then we reached
a contradiction. This contradiction proves that all the poles of $g(z)$
must be inside $(-\infty,0)$, which is equivalent to all the poles
of $\frac{f'(z)}{f(z)}$ are negative. 
\end{proof}

\section{Applications}

\subsection{Riemann $\xi(s)$ function}

Let $s=\sigma+it,\ \sigma,t\in\mathbb{R}$, the Riemann $\xi$-function
is defined by \cite{Davenport,Edwards,Gasper1,Gasper2}
\begin{equation}
\xi(s)=\pi^{-s/2}(s-1)\Gamma\left(1+\frac{s}{2}\right)\zeta(s),\label{eq:3.1}
\end{equation}
where $\Gamma(s)$ and $\zeta(s)$ are the respective analytic continuations
of
\begin{equation}
\Gamma(s)=\int_{0}^{\infty}e^{-x}x^{s-1}dx,\quad\sigma>0\label{eq:3.2}
\end{equation}
 and 
\begin{equation}
\zeta(s)=\sum_{n=1}^{\infty}\frac{1}{n^{s}},\quad\sigma>1.\label{eq:3.3}
\end{equation}
Then $\xi(s)$ is an order $1$ entire function satisfies the functional
equation $\xi(s)=\xi(1-s)$, which implies the Riemann Xi function
$\Xi(s)=\xi\left(\frac{1}{2}+is\right)$ is an even entire function
of order $1$. It is well-known that all the zeros of $\Xi(s)$ are
located within the proper horizontal strip $t\in(-1/2,1/2)$. The
Riemann hypothesis is equivalent to that all the zeros of $\Xi(s)$
are real. 

Since \cite{Davenport,Edwards,Gasper1,Gasper2}
\begin{equation}
\Xi(s)=\int_{-\infty}^{\infty}\Phi(u)e^{ius}du=2\int_{0}^{\infty}\Phi(u)\cos(us)du,\label{eq:3.4}
\end{equation}
where 
\begin{equation}
\Phi(u)=\Phi(-u)=\sum_{n=1}^{\infty}\left(4n^{4}\pi^{2}e^{9u/2}-6n^{2}\pi e^{5u/2}\right)e^{-n^{2}\pi e^{2u}}>0,\label{eq:3.5}
\end{equation}
then, 
\begin{equation}
\xi\left(\frac{1}{2}+s\right)=\sum_{n=0}^{\infty}a_{n}s^{2n},\quad a_{n}=\frac{2}{(2n)!}\int_{0}^{\infty}\Phi(u)u^{2n}du>0.\label{eq:3.6}
\end{equation}
By applying Corollary \ref{cor:3} we obtain the following:
\begin{cor}
The Riemann hypothesis is true.
\end{cor}

\subsection{Character $\xi(s,\chi)$ function }

For a primitive Dirichlet character $\chi(n)$ modulo $q$, let \cite{Davenport,Edwards}
\begin{equation}
\xi(s,\chi)=\left(\frac{q}{\pi}\right)^{(s+\kappa)/2}\Gamma\left(\frac{s+\kappa}{2}\right)L(s,\chi),\label{eq:3.7}
\end{equation}
where $\kappa$ is the parity of $\chi$ and $L(s,\chi)$ is the analytic
continuation of 
\begin{equation}
L\left(s,\chi\right)=\sum_{n=1}^{\infty}\frac{\chi(n)}{n^{s}},\quad\sigma>1.\label{eq:3.8}
\end{equation}
Then $\xi(s,\chi)$ is an entire function of order $1$ such that
\cite{Davenport}
\begin{equation}
\xi(s,\chi)=\epsilon(\chi)\xi(1-s,\overline{\chi}),\label{eq:3.9}
\end{equation}
where 
\begin{equation}
\epsilon(\chi)=\frac{\tau(\chi)}{i^{\kappa}\sqrt{q}},\quad\tau(\chi)=\sum_{n=1}^{q}\chi(n)\exp\left(\frac{2\pi in}{q}\right).\label{eq:3.10}
\end{equation}
Let 
\begin{equation}
G(s,\chi)=\xi\left(s,\chi\right)\cdot\xi\left(s,\overline{\chi}\right),\label{eq:3.11}
\end{equation}
then
\begin{equation}
G(s,\chi)=\epsilon\left(\chi\right)\cdot\epsilon\left(\overline{\chi}\right)G(1-s,\chi).\label{eq:3.12}
\end{equation}
Since \cite{Davenport}
\begin{equation}
\tau\left(\overline{\chi}\right)=\overline{\tau(\chi)},\quad\left|\tau(\chi)\right|=\sqrt{q},\label{eq:3.13}
\end{equation}
then
\begin{equation}
G(s,\chi)=G(1-s,\chi).\label{eq:3.14}
\end{equation}
Since the entire function $\xi\left(\frac{1}{2}+is,\chi\right)$ has
an integral representation, \cite{Davenport}
\begin{equation}
\xi\left(\frac{1}{2}+is,\chi\right)=\int_{-\infty}^{\infty}e^{isy}\varphi\left(y,\chi\right)dy,\label{eq:3.15}
\end{equation}
where
\begin{equation}
\varphi(y,\chi)=2\sum_{n=1}^{\infty}n^{\kappa}\chi(n)\exp\left(-\frac{n^{2}\pi}{q}e^{2y}+\left(\kappa+\frac{1}{2}\right)y\right),\label{eq:3.16}
\end{equation}
then
\begin{equation}
\xi\left(\frac{1}{2}+is,\chi\right)=\sum_{n=0}^{\infty}i^{n}a_{n}(\chi)s^{n},\label{eq:3.17}
\end{equation}
where
\begin{equation}
a_{n}(\chi)=\int_{-\infty}^{\infty}y^{n}\varphi\left(y,\chi\right)dy.\label{eq:3.18}
\end{equation}
It is known that the fast decreasing smooth function $\varphi(y,\chi)$
satisfies the functional equation \cite{Davenport}
\begin{equation}
\varphi(y,\chi)=\frac{i^{\kappa}\sqrt{q}}{\tau\left(\overline{\chi}\right)}\varphi(-y;\overline{\chi}),\quad y\in\mathbb{R},\label{eq:3.19}
\end{equation}
then for all $n\in\mathbb{N}_{0}$,
\begin{equation}
\begin{aligned} & a_{n}(\overline{\chi})=\int_{-\infty}^{\infty}y^{n}\varphi\left(y,\overline{\chi}\right)dy=(-1)^{n}\int_{-\infty}^{\infty}y^{n}\varphi\left(-y,\overline{\chi}\right)dy\\
 & =\frac{(-1)^{n}\tau\left(\overline{\chi}\right)}{i^{\kappa}\sqrt{q}}\int_{-\infty}^{\infty}y^{n}\varphi\left(y,\chi\right)dy=\frac{(-1)^{n}\tau\left(\overline{\chi}\right)}{i^{\kappa}\sqrt{q}}a_{n}(\chi).
\end{aligned}
\label{eq:3.20}
\end{equation}
Let
\begin{equation}
g(s,\chi)=G\left(\frac{1}{2}+is,\chi\right)=\xi\left(\frac{1}{2}+is,\chi\right)\cdot\xi\left(\frac{1}{2}+is,\overline{\chi}\right),\label{eq:3.21}
\end{equation}
then the even entire function $g(s,\chi)$ has the power series expansion,
\begin{equation}
g(s,\chi)=g(-s,\chi)=\sum_{n=0}^{\infty}(-1)^{n}b_{n}(\chi)s^{2n},\label{eq:3.22}
\end{equation}
where $b_{0}(\chi)=\frac{\tau\left(\overline{\chi}\right)}{i^{\kappa}\sqrt{q}}a_{0}^{2}(\chi)$
and $\forall n\in\mathbb{N}$,
\begin{equation}
b_{n}(\chi)=\sum_{j=0}^{2n}a_{j}(\chi)a_{2n-j}(\overline{\chi})=\frac{\tau\left(\overline{\chi}\right)}{i^{\kappa}\sqrt{q}}\sum_{j=0}^{2n}(-1)^{j}a_{j}(\chi)a_{2n-j}(\chi).\label{eq:3.23}
\end{equation}
It is well-known that $\xi(s,\chi)$ is an order $1$ entire function
with infinitely many zeros, all of them are in the horizontal strip
$t\in(-1/2,1/2)$, \cite{Davenport}. Then $g(s,\chi)$ is an order
$1$ even entire function with infinitely many zeros, all of them
are in the horizontal strip $t\in(-1/2,1/2)$. Clearly, all the zeros
of $\xi(s,\chi)$ on the critical line $\sigma=\frac{1}{2}$ if and
only if all the zeros of $g(s,\chi)$ are real. Therefore, the generalized
Riemann hypothesis for $L\left(s,\chi\right)$ is equivalent to that
all the zeros of $g(s,\chi)$ are real. By Corollary \ref{cor:3}
we have the following:
\begin{cor}
The generalized Riemann hypothesis is true for any primitive Dirichlet
character $\chi$ such that 
\begin{equation}
\xi\left(\frac{1}{2},\chi\right)\neq0,\quad\sum_{j=0}^{2n}(-1)^{j}\frac{a_{j}(\chi)}{a_{0}(\chi)}\cdot\frac{a_{2n-j}(\chi)}{a_{0}(\chi)}>0,\quad\forall n\in\mathbb{N},\label{eq:3.24}
\end{equation}
 where $\left\{ a_{n}(\chi)\right\} _{n=0}^{\infty}$ are defined
in (\ref{eq:3.18}).
\end{cor}


\begin{thebibliography}{1}
\bibitem{Gasper2} G. Gasper, Using integrals of squares to prove
that certain real-valued special functions to prove that the P\'{o}lya
$\Xi^{*}\left(z\right)$ function, the functions $K_{iz}\left(a\right),\ a>0$
and some other entire functions have only real zeros, dedicated to
Dan Waterman on the occasion of his 80th birthday.
\end{thebibliography}
\end{document}